\documentclass[11pt]{amsart}
\usepackage{amsthm, amsfonts, amsxtra, amssymb, amscd}

\usepackage{amsfonts,amsmath,amssymb,amscd,amsthm}
\usepackage[english]{babel} 
\usepackage[all]{xy}
 \usepackage{young} \usepackage{ytableau}
\usepackage[backref, colorlinks, linktocpage, citecolor = blue, linkcolor = blue]{hyperref}

\newtheorem*{lemma*}{Lemma}
\newtheorem{lemma}[subsection]{Lemma}
\newtheorem*{theorem*}{Theorem}
\newtheorem{theorem}[subsection]{Theorem}
\newtheorem*{proposition*}{Proposition}
\newtheorem{proposition}[subsection]{Proposition}
\newtheorem*{corollary*}{Corollary}

\newtheorem{corollary}[subsection]{Corollary}

\theoremstyle{definition}
\newtheorem*{definition*}{Definition}
\newtheorem{definition}[subsection]{Definition}
\newtheorem*{example*}{Example}
\newtheorem{example}[subsection]{Example}

\theoremstyle{remark}
\newtheorem*{remark*}{Remark}

\renewcommand{\phi}{\varphi}

\newcommand{\be}{\begin{enumerate}}
\newcommand{\ee}{\end{enumerate}}

\title{Cayley Hamilton algebras and a paper by  Skip Garibaldi}
\author{Claudio Procesi}
\begin{document}\address{Dipartimento di Matematica, G. Castelnuovo,
Universit\`a di Roma La Sapienza, piazzale A. Moro,  00185,
Roma, Italia}

\email{
     procesi@mat.uniroma1.it} \begin{abstract}
In this paper we discuss the minimal Cayley Hamilton norm for a finite dimensional algebra over a field $F$ based on a paper  by  Skip Garibaldi \cite{G}.
\end{abstract}
 \maketitle

    \section{ Introduction}
Since this paper is on Cayley--Hamilton algebras let me recall their formal
definition
\begin{definition}\label{CH}
Given a unital ring $R$ over a commutative ring $A$ a CH--
norm is a polynomial map $N : R  \to A$ (as A--modules) which satisfies the
following 3 properties:\begin{enumerate}\item $N$ is homogeneous of some degree $n$.
\item $N$ is multiplicative, that is $N(ab) = N(a)N(b)$ for all $a, b \in R$
\item Every element $a\in R$ satisfies its abstract characteristic polynomial
   $\chi_a(t) := N(t -a)$.\end{enumerate}
\end{definition}  
 An algebra $R$ equipped with such a norm is called an $n$--Cayley--Hamilton
algebra, CH--algebra for short.\smallskip

Actually in general we have to take the definition of Roby of polynomial map, \cite{Roby} and \cite{Roby1}, so that the previous conditions should be satisfied by $N_B:R_B:=R\otimes_AB\to B$ for every commutative $A$ algebra $B$.\smallskip

Notice that, if $N$ is a CH norm of degree $n$ and $f$ is a multiplicative map
of degree $m$ then $f\cdot   N$ is a CH norm of degree $n + m.$

Of course the main example is for $R := M_n(A)$ and $N$ the determinant.
On the other hand also the powers of $N$ make $R$ into a CH--algebra although
for different degrees.\medskip

This short note treats a special case of a question asked to me by Tian
Cheng Qi, a Ph.D. student at Fudan University. Properly translated his
question can be formulated as follow:\smallskip

{\bf Question}\quad Given an $A$--algebra $R$ which possesses some CH norm $N$ is
there a minimal CH norm $N_0$?

If such $N_0$ exists is it  so that any other CH norm is of the form
$N = f \cdot  N_0$ for $f$ a multiplicative map?\medskip

In this Note I prove that the answer is affirmative when $R$ is  finite dimensional over $F$ a  field, Theorem \ref{main}, and I explain what are the $F$ valued multiplicative
  maps for $R$. In general there may be  several different minimal norms, see example \ref{te}  3).
\smallskip

The proof is in fact a rather direct interpretation of the results
of the nice paper by Skip Garibaldi,\cite{G} {\em  The characteristic polynomial and
determinant are not ad hoc constructions,} whose goal was to define an intrinsic determinant for an algebra $R.$ It turns out that his definition gives
the canonical norm $N_0$. In fact the method is based on the clasical approach of Dedekind and Frobenius  to characters of finite groups.\medskip

The note is the result of a lucky and fortuitous combination of independent events. As I was thinking about the previous question I came across a
discussion in MathOver
flow Newsletter - Friday, August 30, 2024 

{\em Are automorphisms of matrix algebras necessarily determinant preservers?}
\smallskip

In one of the answers Qiaochu Yuan \cite{Q} pointed out the paper by Skip
Garibaldi which contains the main ingredients for the proof of the Theorem.
\section{The paper by Garibaldi, \cite{G}.}
\subsection{ Multiplicative polynomials}
Let $R$ be an algebra with 1, not necessarily associative,  finite dimensional
over an infinite  field $F$. Set $m := \dim_F R$.

Consider the algebra of polynomial functions (with values in $F$) on $R$
(as vector space). 

Choosing a basis $a_1,\cdots,a_m $ of $R$ this  algebra is isomorphic to the algebra $A :=
F[x_1,\cdots,x_m]$ of polynomials in $m$ variables. The identification is by computing $f$ on the {\em generic element}
\begin{equation}\label{ge}
x :=\sum^m_{i=1}x_ia_i \in  R\otimes_ F A .
\end{equation} 
\begin{definition}\label{mu}
A polynomial function $f : R \to F$ is multiplicative if $f(ab) =
f(a)f(b)$ for all $a,b \in R$.
\end{definition}   
\begin{lemma}\label{ho}
If $f $ is multiplicative then it is homogeneous of some degree $k$
and $f(1) = 1$.
\end{lemma}  \begin{proof}
It is well known that for $R = F$ the multiplicative polynomial maps
are just the powers   $\alpha\mapsto\alpha^ k$ so there is a $k$ such that $f$ restricted to $F$ is
   $\alpha\mapsto\alpha^ k$ hence $f(1)=1$. Now for all $a \in R$ and    for all  $\alpha\in F$ we have $f(\alpha   a) = f(\alpha   )f(a) = \alpha^   kf(a)$
so $f$ is homogeneous of degree $k$.
\end{proof}  
Given a multiplicative polynomial function $f$ decompose $f =
\prod^h_{j=1} f_j$
with $f_j$ irreducible.
Since $1 = f(1) =
\prod^h
_{j=1} f_j(1)$ we can normalyse the factors replacing $f_j$
with $f_j(1)^{-1}f_j$ so that $f_j(1) = 1$ for all $j.$ Given two generic elements
$$x:=\sum^m_{i=1}x_ia_i,\ y:=\sum^m_{i=1}y_ia_i$$
we have
$$\prod^h_{j=1} f_j(x)f_j(y) = f(x)f(y) = f(xy) =
\prod^h_{j=1} f_j(xy).$$
By the unique factorization it follows that $f_j(xy) = u_j(x)v_j(y)$ for two
polynomials a priori product of factors of the previous factorization, again
we can normalyze so that $u_j(1) = v_j(1) = 1$. Now evaluating $y \mapsto 1$ we have
$f_j(x \cdot  1) = u_j(x); f_j(1\cdot   y) = v_j(y)$ so we have proved.
\begin{proposition}\label{imu}
Consider a decomposition of polynomial maps $f = g \cdot  h$
with $g(1) = h(1) = 1$. Then $f$ is multiplicative if and only if both $g$ and $h$
are multiplicative.
\end{proposition}   
We now assume that $R$ is an associative algebra. \footnote{For non--associative but power associative algebras there are results by Jacobson see the rferences in \cite{G}.
}
\subsection{ The minimal polynomial}
Now the connection to CH--algebras.

Consider the generic element \eqref{ge}, $x :=\sum^m_{i=1}x_ia_i \in  R\otimes_ F A \subset R\otimes_ FK$ with
$K := F(x_1,\cdots,x_m)$ the function  field.

Since $R \otimes_F K$ is  finite dimensional over the  field $K$ we may consider
the minimal polynomial $P(t)  \in  F(x_1,\cdots,x_m)[t]$ of the generic element $x :=\sum^m_{i=1}x_ia_i $  over $K$.
\begin{proposition}\label{int}
\begin{enumerate}\item The polynomial $P(t)$ has coefficients in $A = F[x_1,\cdots,x_m]$
the polynomial ring.
\item If $k$ is the degree of $P$ the polynomial $f (x_1,\cdots,x_m):= (-1)^kP(0)$ is multiplicative.
\item $P(t) = f(t -x)$.
\item $f$ is a CH--norm.\end{enumerate}
\end{proposition}  
\begin{proof}
1) It is enough to prove that $x$ satisfies a monic polynomial $Q(t)$ with
coefficients in $A = F[x_1,\cdots,x_m]$ since then $P$ is a factor of $Q$ and we can
apply Gauss Lemma.

If one takes any faithful representation of $R$ as matrices, for instance
the regular representation, the generic element $x$ becomes a matrix $X$ with
entries linear polynomials in the $x_i$ so we can take for $Q(t)$ its characteristic
polynomial $\det(t -X)$.

2), 3) The polynomial $\det(X)$ is multiplicative. Decompose it into its
irreducible (and multiplicative by Proposition 2.4) factors $\det(X) =\prod_i f_i(x)$
so that $Q(t) =\prod_i f_i(t-x)$. Now if $g(x)$ is irreducible, as polynomial in the
$x_i$ it is also irreducible in $A[t] = F[t, x_1,\cdots,x_m]$ and so also $g(t -x)$ is
irreducible, therefore $\det(t-X) =
\prod_i f_i(t-x)$ is the decomposition into
irreducible factors.

Moreover if $k$ is the degree of $g$ we have $(-1)^kg(x) = g(0-x)=P(0)$. 

Since $P(t)$
is a factor of $Q(t)$ it is a product of some of its irreducible factors $f_i(t-x)$
and both claims follow.

4) This follows from the definition.
\end{proof}    
\subsection{ The main Theorem}
We can now prove our main result
\begin{theorem}\label{main}
The polynomial $f := (-1)^kP(0)$ is a unique minimal CH--
norm. All other CH--norms are of the form $g \cdot  f$ for $g$ multiplicative.
\end{theorem}   \begin{proof}
Let $g$ be a CH--norm, so $Q(t) = g(t-x)$ is a monic polynomial
satisfied by $x$, hence it is a multiple of $P(t)$ and the claim follows by the
previous Proposition \ref{imu}.
\end{proof}  
\begin{definition}\label{der}
The degree of the minimal CH norm of $R$ will be called the {\em degree} of $R$.
\end{definition}
\begin{example}\label{du}

\end{example}   Let $R = F + F\epsilon,\ \epsilon^2 = 0$ be the dual numbers. The generic
element $x_1 + x_2\epsilon$  satisfies:
$$(x_1 + x_2\epsilon)^2 = x_1^2
+ 2x_1x_2  \epsilon$$
so its minimal polynomial is $(t -x_1)^2 = t^2-2x_1t + x_1^2$.
 The minimum
CH--norm is $x^2_1$.
 
For $n\times  n$ matrices the determinant is irreducible of degree $n$ and it is the minimal
CH--norm (by the next Proposition  \ref{deg}).

For the real quaternions it is the usual quadratic norm.\subsection{Complements}
Given $R$ finite dimensional over a field $F$  the degree $n$  of its minimal CH norm  has the following interpretation.
\begin{proposition}\label{deg}
Every element $a\in R$  satisfies a polynomial of degree $n$ over $F$  and there is a non empty Zariski open set $U$ of $R$  such that each element of $U$  does not satisfy any polynomial of degree $<n$.
\end{proposition}
\begin{proof}
The first part is clear  if $a=\sum_i\alpha_ia_i$  then $a$ satisifies the polynomial $N(t-a)$  obtained from $P(t)$  by specializing $x_i\mapsto \alpha_i$.

An element $a$  satisfies a polynomial of degree $<n$  if and only if the $n$ elements  $a^i,\ i=0,\cdots,n-1$ are linearly dependent.

Now  take the $n$ elements  $x^i=\sum_{j=1}^m f_{i,j}  (x_1,\cdots,x_m)a_j,\ i=0,\cdots,n-1$\ where $x$ is the generic element. Since $x$ does not satisfy any polynomial of degree $n-1$  with coefficients in $K$ the $n\times m$  matrix with entries the polynomials $f_{i,j}  (x_1,\cdots,x_m)$  has rank $n$  so some of its $n\times n$ minors have determinant different from 0. The subvariety of $R$ where the elements have degree $<n$  is thus given by the vanishing of all these determinants. \end{proof}
As a corollary we have
\begin{corollary}\label{sub}
If $R\subset S$  are two finite dimensional algebras of the same degree then the minimal CH norm of $S$  restricted to $R$ is the minimal CH norm of $R.$
\end{corollary}
In turn this  implies that Theorem \ref{main} holds also for a class of infinite dimensional algebras.
Let $R$  be an algebra over $F$  and algebraic over $F$,  its degree is the sup of the degrees of its elements. This of course can be infinity  but if it is some $n<\infty$  we have  that Theorem \ref{main} holds for $R$. In fact it is well  known cf. \cite{agpr}  that $R$ is {\em locally finite }  which easily implies  that $R=\bigcup_iR_i$  where each $R_i$ is finite dimensional of degree $n$. Then the previous corollary gives the claim.\medskip

If the degree is infinite several pathologies can occur, it need not be locally finite by a Theorem of Golod--Shafarevich  see \cite{agpr}. Even if it is locally finite and simple  it need not have any multiplicative map. In fact the sequence of algebras $M_n(F)^{\otimes k},\ k=1,\cdots$  can be embedded one in the other by mapping $a\mapsto a\otimes 1_n$ and then the statement is clear.\medskip

It remains the general question of algebras over a commutative ring $A$. If such an algebra $R$ has an $n$--CH norm then it is integral over $A$  of degree $n$ and still locally finite. So is this condition sufficient?

In this generality certainly not. A the following  examples point out to some further restrictions.

\begin{example}\label{te}

1)\quad Let $A$  be a domain containing an infinite field $F$ and $B$ its integral closure, assume $B\neq A$.  If there is a norm $N:B\to A$  of degree $n$  take $a\in B\setminus A$ then there exist $b,c\in A$ with $ba=c$ so that $c^n=a^nN(a)\implies N(a)=a^n\in A$.

Now take $n+1$  distinct elements $f_i\in F$ and we have
$$ (a+f_i)^n=a^n+\sum_{j=1}^n \binom nif_i^ja^{n-j}\in A.$$ This implies that setting $f$ the Vandermonde determinant of the $f_j$  we have $n\cdot f\cdot a\in A$  contradicting our hypotheses $a\notin A$.\smallskip

2)\quad Let $A$  be a domain and      $B $ an $A$ algebra,  let $T\subset B$ be the set of torsion  elements, then $T$ is an ideal  of  $B $ and $B/T$ is torsion free.  

 If there is a norm $N:B\to A$  of degree $n$   we claim that $N$ factors through $T$.  In fact  compute $N(a +b),\ b\in T$.  We have $sb=0$  for some $s\in A\setminus\{0\}$  so
$$s^nN(a +b)=N(sa)=s^nN(a)\implies N(a+b)=N(b). $$  In particular if $N$ is a CH norm then every element of $T$ is nilpotent of degree $\leq n$. The same holds for $A$ any commutative ring and the torsion of $R$ with respect to the multiplicative set of the non zero divisors of $A$.\smallskip

3)\quad Let us take the two $F$ algebras $M_3(F),\ F\oplus F$  and consider their direct sum $R:=M_3(F)\oplus( F\oplus F)$  as an $F\oplus F$ algebra.
Then $R$ has degree 3 over $F\oplus F$   and two different  CH norms of degree 3.
$$ N_1(A,(x,y)):=(\det(A), x^2y ),\quad N_2(A,(x,y)):=(\det(A), xy^2 ).$$ Of course by the same method we may construct different minimal CH norms for any degree $n>2$.  We will see that for $n=2$ this is not possible.
\end{example}

This suggests to restrict to analyse the case $A$ an integrally closed domain and $R$  torsion free and integral of some degree $n$.

Then $R\subset S:=R\otimes_AK$ with $K$ the field of fractions of $A$. Any $A$ valued norm on $R$ extends to a $K$ valued norm on $S$.            The first question is to see if the canonical CH norm $N$ of $S$  restricted to $R$ takes values in $A$.

 If $A$ is a UFD and $R$ is free over $A$  then the argument of Theorem \ref{main} holds verbatim. 
 

 In fact it is enough to assume $A$  integrally closed and Noetherian  by localization to the minimal primes which are DVR hence UFD and finitely generated torsion free modules over a DVR are free.  
 
 If $A$ contains an infinite field $F$ the hypothesis of being Noetherian can be dropped. In fact
this question reduces immediately to $R$ commutative,  take $a\in R$  and let $B$ be the subalgebra  of $R$  generated by $a$ which is a finite module over $A$.  Then one can take the algebra over $F$ finitely generated by the coefficients of a monic  polynomial of degree $n$ satisfied by $a$, its integral closure is still contained in $A$  since $A$ is integrally closed and it is Noetherian by Noether's normalization.  This reduces to the previous case.\smallskip
%
%

 It may have some interest to pursue further this  analysis.

\section{The multiplicative maps}
\subsection{The reduced norm}
What can we say about the possible multiplicative maps of a  finite dimensional $F$ algebra $R$?

The  first step is a reduction to semisimple algebras (in fact in positive
characteristic it is better to assume separable).

One can start from a more general question let $ R^\star$ be the multiplicative
group of $R$, a multiplicative polynomial restricted to $ R^\star$ is an algebraic
homomorphism to $ F^\star$ that is a {\em character}. So we reformulate the problem:
What can we say about the characters of $ R^\star$?

It is well known and easy to prove that for an infinite  field $F$ the characters of the group $GL(n; F)$ of invertible $n \times  n$ matrices are all integral
powers $\det(X)^k,\ k \in\mathbb Z$ and are polynomial if and only if $k\in\mathbb N.$

Let $J$ be the radical of $R$, we have that $1 + U$ is a unipotent normal
subgroup of $ R^\star$ and
$ R^\star/(1 + U) = (R/J)^\star.$ 
A character is trivial on a unipotent subgroup so we have that the characters
of $ R^\star$ coincide with those of $(R/J)^\star.$

\begin{lemma}\label{au}
The multiplicative maps of the  finite dimensional algebra $R$
are induced by the polynomial characters of  $(R/J)^\star.$
\end{lemma}  
Assume $R$ semisimple, and in positive characteristic separable. Let $\bar F$ be the algebraic closure of $F$.
Then $R 
\otimes_F  \bar F = \oplus_jM_{h_j} ( \bar  F)$ is a direct sum of matrix algebras over   $\bar  F$
permuted by the Galois group $G$ (in fact one may take instead of   $\bar  F$  a  finite
Galois extension). As pointed out in \cite{G}, the generic element $x$ for $R$ is also a
generic element $x$ for  $R 
\otimes_F  \bar F$ and its minimal polynomial is the same $P(t).$

The canonical CH norm of $R 
\otimes_F  \bar F = \oplus_   jM_{h_j} ( \bar  F)$ is the product of all
determinants which, being permuted by the Galois group $G$, is of course a
polynomial with coefficients in $F$ as we already know. But in fact if we
group these determinants into the orbits under $G$ we have the irreducible
multiplicative polynomials of $R$.

In explicit terms $R =\oplus_  iR_i$ with $R_i$ simple. Clearly a multiplicative
polynomial map on $R$ is just the product of multiplicative polynomial maps
on $R_i$ for each $i$. So now assume $R$ simple, this means that $R = M_n(D)$ the
algebra of $n\times   n$ matrices over a division ring $D $   finite dimensional over $F$.
Let $Z$ be the center of $D$, I will discuss only the case $Z$ separable over $F$,
then this is classical material which I include for convenience of the reader.

Let $h := \dim_F Z,\ \dim_Z D = p^2$ and let  $\bar Z$ be an algebraic closure of $Z$
(and $F$) consider
$$M_n(D)\otimes_F  \bar Z = M_n(D
\otimes_F  \bar Z) = M_n(D\otimes _Z(Z\otimes _F \bar  Z) = M_n(D
\otimes _Z\bar  Z)^h = M_{pn}( \bar  Z)^h.$$
A multiplicative function $f:   M_{pn}( \bar  Z)^h.\to\bar   Z$ is of the form $$(z_1,\cdots      ,z_h) \mapsto\prod^h
_{i=1} \det(z_i)^{\ell_i}$$ and if it is Galois invariant all $\ell_i$ are equal to some $\ell$.

When $\ell= 1$ this map, restricted to $R = M_n(D)$, takes values in $F$ and it  is called the {\em  reduced
norm} and denoted by $N_{R/F }$. Any other multiplicative map is thus a power
of the reduced norm. In conclusion.

\begin{theorem}\label{fin}
The irreducible multiplicative polynomial maps for an alge-bra $R$ are the reduced norms of the simple factors of $R/J$.

The product of these irreducibles is the canonical CH norm of $R/J$.
\end{theorem} 
\subsection{ Conclusions}
In general the minimal CH--norm of $R$ is a multiple of that of $R/J $ so it is a
product of these norms with positive exponents, but in general  the exponents may be bigger than 1. 

They are numerical invariants of the algebra
$R$ which would be interesting to investigate, one point is to understand the
relation between $R/J$ and the radical, for instance if $R$ is separable then
$R =  R/J\oplus J $ as vector space with $R/J =\oplus _ iR_i$ a subalgebra (Wedderburn).
The exponents may be related to the nilpotent degree of $R_iJ$.

 If $F$ is algebraically closed and $R$ is semisimple it is the direct sum $\oplus_   iEnd(V_i)$ and then the CH norm is the determinant of its minimal faithful representation
$\oplus_   iV_i$, this is certainly not true for non semisimple case for instance take $R$
to be the polynomial algebra in $n > 1$ variables truncated at degree 1 (all 
monomials of degree  $\geq  2$ set to 0). The degree of the canonical CH norm is
2 but this has no faithful 2--dimensional representation over $F$.

  In characteristic 0 one can develop the theory using the trace associated to the norm, and one knows \cite{agpr} that, if $n = tr(1)$, there is a canonical embedding
of $R$ into $n\times   n$ matrices over a commutative ring $B$ compatible with norms. Maybe one
can relate some properties of $B$ to these invariants.
\begin{example}\label{ba}
 If $R/J = F$, the minimal CH--norm is of the form $x^k; \ x \in F.$
 
 The exponent $k$  is the maximal degree of nilpotency of elements of $J$ (cf. 2.9).
\end{example}  
At this point one may try to treat the general case. If $R$ is a  finite free
or maybe even projective module over a commutative ring $A$ the same ideas
may work. If $R$ is Azumaya with center $A$ then it has a canonical reduced norm and this
is the minimal CH--norm with values in the center and by Theorem 2.28 of \cite{pv}  all other
norms are a power of this by faithfully 
flat descent.

 The general case may be
impossible to treat but maybe some more cases can be treated.\medskip
 
We finish by showing  that  an  algebra $R$   over $A$ with 2 invertible cannot have two different quadratic $A$ valued CH norms. 

In this case the  norm is induced by the trace by the formula  $2N(x)=tr(x)^2-tr(x^2)$.\smallskip

In fact suppose we had two such norms $N_1, N_2$ and for some element $x\in R$ we have $N_1(x)=a, N_2(x)=b,\ a\neq b$. Let us consider the two formal characteristic polynomials for $x$ induced by the two norms \begin{equation}\label{tc}
N_1(t-x)=t^2-ct+a, \ tr_1(x)=c;\ N_2(t-x)=t^2-dt+b,\ tr_2(x)=d.
\end{equation}
For $N_1$ the  characteristic polynomial of $ux$ is $t^2-uct+u^2a$  (by the formula with trace) so the  characteristic polynomial of $ux+v$ is 
\begin{equation}\label{c1}
(t-v)^2-vc(t-v)+v^2a=t^2-(2v+vc)t+v^2c+u^2a 
\end{equation} For $N_2$ the  characteristic polynomial of $ux$ is $t^2-udt+u^2b$    so the  characteristic polynomial of $ux+v$ is 
\begin{equation}\label{c2}
(t-v)^2-vd(t-v)+v^2b=t^2-(2v+vd)t+v^2d+u^2b.
\end{equation} By CH  and \eqref{tc}:$$ x^2-cx+a=x^2-dx+b=0\implies  (d-c)x+a-b=0,\  d-c,a-b\neq 0.$$So for $u=d-c,\ v=a-b$ we have both polynomials \eqref{c1}, \eqref{c2},  become $=t^2$ hence:
 $$\implies  (2+c)(a-b)= (2+d)(a-b)=0,\quad  (a-b)^2c+(d-c)^2a=(a-b)^2d+(d-c)^2b=0.$$\ In fact the last follows from the first 3.
 
The norm must factor modulo        $ (d-c)x+a-b$     
$$N_1(ux+v+ (d-c)x+a-b)= N_1(ux+v)\implies v^2c+u^2a=(v+a-b)^2c+(u+d-c)^2a$$

$$\implies  2v(a-b)c+2u(d-c)a=0,\ \forall u,v\implies 2 (a-b)c=2 (d-c)a=0.$$$$(2+c)(a-b)= 2 (a-b)c=9,\implies -4(a-b)=0$$
 which implies $a=b$ a contradiction.
  
  \bibliographystyle{amsalpha}

\end{document}